\documentclass[12pt,reqno]{amsart}
\usepackage{amsfonts}
\usepackage{amssymb, latexsym}
\usepackage{eufrak}

\usepackage{amsmath,amsfonts,epsfig}
\usepackage{amsthm}
\usepackage{graphicx}
\usepackage{graphics,color}
\usepackage{amsbsy}
\usepackage{amssymb}

\setbox0=\hbox{$+$}
\newdimen\plusheight
\plusheight=\ht0
\def\+{\;\lower\plusheight\hbox{$+$}\;}

\setbox0=\hbox{$-$}
\newdimen\minusheight
\minusheight=\ht0
\def\-{\;\lower\minusheight\hbox{$-$}\;}

\setbox0=\hbox{$\cdots$}
\newdimen\cdotsheight
\cdotsheight=\plusheight
\def\cds{\lower\cdotsheight\hbox{$\cdots$}}

\renewcommand{\(}{\left\(}
\renewcommand{\)}{\right\)}
\renewcommand{\[}{\left[}

\numberwithin{equation}{section}
 \theoremstyle{plain}
\newtheorem{theorem}{Theorem}[section]
\newtheorem{lemma}[theorem]{Lemma}

\newtheorem{corollary}[theorem]{Corollary}

\begin{document}
\allowdisplaybreaks
\title[New Congruences Modulo 2, 4, and 8 for the Number of Tagged Parts Over the Partitions with Designated Summands] {New Congruences Modulo 2, 4, and 8  for the Number of Tagged Parts Over the Partitions with Designated Summands}

\author{Nayandeep Deka Baruah}
\address{Department of Mathematical Sciences, Tezpur University, Napaam-784028, Sonitpur, Assam, INDIA}
\email{nayan@tezu.ernet.in}

\author{Mandeep Kaur}
\address{Department of Mathematical Sciences, Tezpur University, Napaam-784028, Sonitpur, Assam, INDIA}
\email{mandeep@tezu.ernet.in}


\vspace*{0.5in}
\begin{center}
{\bf New Congruences Modulo 2, 4, and 8 for the Number of Tagged Parts Over the Partitions with Designated Summands}\\[5mm]
{\footnotesize  NAYANDEEP DEKA BARUAH and MANDEEP KAUR
}\\[3mm]
\end{center}

\vskip 5mm \noindent{\footnotesize{\bf Abstract.}
Recently, Lin introduced two new partition functions $\textup{PD}_\textup{t}(n)$ and $\textup{PDO}_\textup{t}(n)$, which count the total number of tagged parts over all partitions of $n$ with designated summands and the total number of tagged parts over all partitions of $n$ with designated summands in which all parts are odd. Lin also proved some congruences modulo 3 and 9 for $\textup{PD}_\textup{t}(n)$ and $\textup{PDO}_\textup{t}(n)$, and conjectured some congruences modulo 8. In this paper, we prove the congruences modulo 8 conjectured by Lin and also find many new congruences and infinite families of congruences modulo some small powers of 2.}

\vskip 3mm

\noindent {\footnotesize 2010 Mathematical Reviews Classification
Numbers: Primary 11P83; Secondary 05A17}.

\section{\textbf{Introduction}}\label{secone}
In \cite{lovejoy}, Andrews, Lewis and Lovejoy introduced and studied a new class of partitions, partitions with designated summands. Partitions with designated summands are constructed by taking ordinary partitions and tagging exactly
one of each part size. For example, there are 10 partitions of $4$
with designated summands, namely,
$$4^\prime, \quad 3^\prime+1^\prime,  \quad 2^\prime+2, \quad 2+2^\prime, \quad 2^\prime+1^\prime+1, \quad
2^\prime+1+1^\prime,$$
 $$\quad  1^\prime+1+1+1, \quad 1+1^\prime+1+1, \quad 1+1+1^\prime+1, \quad 1+1+1+1^\prime.$$
The total number of partitions of $n$ with designated summands is denoted by $\textup{PD}(n)$. Hence, $\textup{PD}(4)=10.$ Andrews, Lewis and Lovejoy \cite{lovejoy} also studied $\textup{PDO}(n)$, the number of partitions of $n$ with designated summands in which all parts are odd. From the above example, $\textup{PDO}(4)=5$. Further studies on $\textup{PD}(n)$ and $\textup{PDO}(n)$ were carried out by Chen, Ji, Jin,  and Shen \cite{chen-ji}, Baruah and Ojah \cite{ndb-ojah-integers}, and Xia \cite{xia-jnt}.

Recently, Lin \cite{lin} introduced two new partition functions $\textup{PD}_\textup{t}(n)$ and $\textup{PDO}_\textup{t}(n)$, which count the total number of tagged parts over all partitions of $n$ with designated summands and the total number of tagged parts over all partitions of $n$ with designated summands in which all parts are odd, respectively. From the partitions of $4$ with designated summands given above, we note that  $\textup{PD}_\textup{t}(4)=13$ and $\textup{PDO}_\textup{t}(4)=6$.
Lin \cite{lin}  proved that the generating functions of $\textup{PD}_\textup{t}(n)$ and $\textup{PDO}_\textup{t}(n)$ are
\begin{align}\label{pdt-gen}
\sum_{n=0}^{\infty}\textup{PD}_\textup{t}(n)q^n&=\dfrac{1}{2}\left(\dfrac{f_3^5}{f_1^3f_6^2}-\dfrac{f_6}{f_1f_2f_3}\right)\\\intertext{and}
\label{pdot-gen}\sum_{n=0}^{\infty}\textup{PDO}_\textup{t}(n)q^n&=\dfrac{qf_2f_3^2f_{12}^2}{f_1^2f_6},
\end{align}
where as usual, for any complex number $a$ and $|q|<1$,
\begin{align*}
(a;q)_\infty&:=\prod_{n=1}^\infty(1-aq^{n-1})
\end{align*}
and for any positive integer $k$, $f_k:=(q^k;q^k)_\infty$.

Lin \cite{lin} also derived several congruences modulo small powers of 3 for $\textup{PD}_\textup{t}(n)$ and $\textup{PDO}_\textup{t}(n)$. For example, for any nonnegative integers $n$ and $k$,
\begin{align}
\textup{PD}_\textup{t}(3n)&\equiv 0 ~(\textup{mod}~3),\notag\\
\textup{PD}_\textup{t}(3n+2)&\equiv 0 ~(\textup{mod}~3),\notag\\
\label{pdt-3-lin36-21}\textup{PD}_\textup{t}(36n+21)&\equiv 0 ~(\textup{mod}~9),\\
\label{pdt-3-lin36-33}\textup{PD}_\textup{t}(36n+33)&\equiv 0 ~(\textup{mod}~9),\\
\textup{PD}_\textup{t}(48n+20)&\equiv 0 ~(\textup{mod}~9),\notag\\
\textup{PD}_\textup{t}(48n+36)&\equiv 0 ~(\textup{mod}~9),\notag\\
\textup{PD}_\textup{t}(72n+42)&\equiv 0 ~(\textup{mod}~9),\notag\\
\textup{PD}_\textup{t}(72n+66)&\equiv 0 ~(\textup{mod}~9),\notag\\
\textup{PDO}_\textup{t}(8n)&\equiv 0 ~(\textup{mod}~9),\notag\\
\textup{PDO}_\textup{t}(24n)&\equiv 0 ~(\textup{mod}~27),\notag\\
\textup{PDO}_\textup{t}(36n)&\equiv 0 ~(\textup{mod}~27),\notag\\
\textup{PDO}_\textup{t}(36n+24)&\equiv 0 ~(\textup{mod}~27),\notag\\
\textup{PDO}_\textup{t}(8\cdot5^{2k+1}(30n+6a+5))&\equiv 0 ~(\textup{mod}~27),\notag
\end{align}
where $a=1,2,3,4$.

Very recently, Adansie, Chern and Xia \cite{adansie-etal} found the following two infinite families of congruences modulo 9.

For any nonnegative integers $n$ and $k$,
\begin{align*}
\textup{PD}_\textup{t}(3^{2k+1}(9n+2))&\equiv 0 ~(\textup{mod}~9)\\\intertext{and}
\textup{PD}_\textup{t}((3^{2k+1}(9n+7))&\equiv 0 ~(\textup{mod}~9).
\end{align*}

By analyzing a large number of values of $\textup{PD}_\textup{t}(n)$ and $\textup{PDO}_\textup{t}(n)$ via MAPLE, Lin \cite{lin} further speculated the existence of congruences modulo small powers of 2. For example,  he conjectured that, for any nonnegative integer $n$,
\begin{align}
\label{pdt-lin1}\textup{PD}_\textup{t}(48n+28)&\equiv 0 ~(\textup{mod}~8),\\
\label{pdt-lin2}\textup{PD}_\textup{t}(48n+46)&\equiv 0 ~(\textup{mod}~8),\\
\label{pdot-lin1}\textup{PDO}_\textup{t}(8n+6)&\equiv0 ~(\textup{mod}~8),\intertext{and}
\label{pdot-lin2}\textup{PDO}_\textup{t}(8n+7)&\equiv0 ~(\textup{mod}~8).
\end{align}
In this paper, we prove the above congruences and also find many new congruences and infinite families of congruences modulo 2 and 4.

 The following theorem states the exact generating functions of $\textup{PDO}_\textup{t}(8n+6)$ and $\textup{PDO}_\textup{t}(8n+7)$ that immediately implies the congruences \eqref{pdot-lin1} and \eqref{pdot-lin2}.

\begin{theorem}\label{thm-lin-pdot}
For any nonnegative integer $n$, we have
\begin{align}
\label{thm-pdot-lin1}\sum_{n=0}^{\infty}\textup{PDO}_\textup{t}(8n+6)q^n&=8\left(2\dfrac{f_2^{16}f_6^{10}}{f_1^{17}f_3^3f_{12}^4}-q
\dfrac{f_2^{28}f_3f_{12}^{4}}{f_1^{21}f_6^2f_{4}^8}-16q^2\dfrac{f_2^{4}f_3f_4^{8}f_{12}^4}{f_1^{13}f_6^2}\right)\\\intertext{and}
\label{thm-pdot-lin2}\sum_{n=0}^{\infty}\textup{PDO}_\textup{t}(8n+7)q^n&=8\left(\dfrac{f_2^{14}f_3f_6^{4}f_8^{2}}{f_1^{14}f_4^3f_{12}^2}+2
\dfrac{f_2^{9}f_3^2f_{4}^5f_6}{f_1^{13}f_8^2}+4q\dfrac{f_2^{8}f_3^3f_4f_8^2f_{12}^2}{f_1^{12}f_6^2}\right).\end{align}
\end{theorem}

In the following theorem and corollary, we present our new congruences and infinite families of congruences modulo 2 and 4 for $\textup{PD}_\textup{t}(n)$.

\begin{theorem}\label{thm-mod2}
For any nonnegative integers $k$, $\ell$  and $n$, we have
\begin{align}
\label{24-12}\textup{PD}_\textup{t}(24n+12)&\equiv 0~(\textup{mod}~2),\\
\label{24-21}\textup{PD}_\textup{t}(24n+21)&\equiv 0~(\textup{mod}~2),\\
\label{48n30}\textup{PD}_\textup{t}(48n+30)&\equiv 0~(\textup{mod}~2),\\
\label{eq133}\textup{PD}_\textup{t}(144n+102)&\equiv 0~(\textup{mod}~2),\\
\label{216-153}\textup{PD}_\textup{t}(216n+153)&\equiv 0~(\textup{mod}~2),\\
\label{36n21}\textup{PD}_\textup{t}(36n+21)&\equiv 0~(\textup{mod}~4),\\
\label{36n33}\textup{PD}_\textup{t}(36n+33)&\equiv 0~(\textup{mod}~4),\\
\label{4k.12n}\textup{PD}_\textup{t}(2^{2k}\cdot12n)&\equiv \textup{PD}_\textup{t}(12n) ~(\textup{mod}~4),\\
\label{24-l-12}\textup{PD}_\textup{t}(3^\ell\cdot2^{2k}(24n+12))&\equiv \textup{PD}_\textup{t}(24n+12) ~(\textup{mod}~4),\\
\label{96-60}\textup{PD}_\textup{t}(96n+60)&\equiv 0~(\textup{mod}~4),\\
\label{96-84}\textup{PD}_\textup{t}(96n+84)&\equiv 0~(\textup{mod}~4),\\
\label{144-84}\textup{PD}_\textup{t}(144n+84)&\equiv 0~(\textup{mod}~4),\\
\label{144-120}\textup{PD}_\textup{t}(144n+120)&\equiv 0~(\textup{mod}~4),\\
\label{144-132}\textup{PD}_\textup{t}(144n+132)&\equiv 0~(\textup{mod}~4),\\
\label{288-204}\textup{PD}_\textup{t}(3^k(288n+204))&\equiv \textup{PD}_\textup{t}(288n+204) \equiv 0~(\textup{mod}~4),\\
\label{864-792}\textup{PD}_\textup{t}(864n+792)&\equiv 0~(\textup{mod}~4),\\
\label{1728-1224}\textup{PD}_\textup{t}(1728n+1224)&\equiv 0~(\textup{mod}~4),\\
\label{2592-1080}\textup{PD}_\textup{t}(2592n+1080)&\equiv 0~(\textup{mod}~4),\\
\label{36-30}\textup{PD}_\textup{t}(36n+30)&\equiv 0~(\textup{mod}~4),\\
\label{108-90}\textup{PD}_\textup{t}(108n+90)&\equiv 0~(\textup{mod}~4),\\
\label{12-6}\textup{PD}_\textup{t}(3^{2k}(12n+6))&\equiv \textup{PD}_\textup{t}(12n+6)~(\textup{mod}~4).
\end{align}
\end{theorem}

\begin{corollary}\label{cor-mod4-1}
For any positive integers $k$, $\ell$ and any nonnegative integer $n$, we have
\begin{align*}
\textup{PD}_\textup{t}(3^\ell\cdot2^{2k}(8n+5))&\equiv 0~(\textup{mod}~4),\\
\textup{PD}_\textup{t}(3^\ell\cdot2^{2k}(8n+7))&\equiv 0~(\textup{mod}~4),\\
\textup{PD}_\textup{t}(3^\ell\cdot2^{2k}(12n+7))&\equiv 0~(\textup{mod}~4),\\
\textup{PD}_\textup{t}(3^\ell\cdot2^{2k}(12n+11))&\equiv 0~(\textup{mod}~4),\\
\textup{PD}_\textup{t}(3\cdot2^{2k+1}(6n+5))&\equiv 0~(\textup{mod}~4),\\
\textup{PD}_\textup{t}(3^{\ell+1}\cdot2^{2k}(24n+17))&\equiv 0~(\textup{mod}~4),\\
\textup{PD}_\textup{t}(3^2\cdot2^{2k+1}(12n+11))&\equiv 0~(\textup{mod}~4),\\
\textup{PD}_\textup{t}(3^2\cdot2^{2k+1}(24n+17))&\equiv 0~(\textup{mod}~4),\\
\textup{PD}_\textup{t}(3^3\cdot2^{2k+1}(12n+5))&\equiv 0~(\textup{mod}~4),\\
\textup{PD}_\textup{t}(2\cdot3^{k}(6n+5))&\equiv 0~(\textup{mod}~4).
\end{align*}
\end{corollary}

\begin{proof}
Congruences \eqref{96-60}-- \eqref{144-84} and \eqref{144-132} may be rewritten as
\begin{align*}
\textup{PD}_\textup{t}(24(4n+2)+12)&\equiv 0~(\textup{mod}~4),\\
\textup{PD}_\textup{t}(24(4n+3)+12)&\equiv 0~(\textup{mod}~4),\\
\textup{PD}_\textup{t}(24(6n+3)+12)&\equiv 0~(\textup{mod}~4),\intertext{and}
\textup{PD}_\textup{t}(24(6n+5)+12)&\equiv 0~(\textup{mod}~4),
\end{align*}
respectively. From \eqref{24-l-12} and the above congruences, we easily arrive at the first four infinite families of congruences of the corollary. Since the other congruences can also be proved in a similar way, we omit the details.
\end{proof}


We organize the rest of the paper as follows. In Section \ref{sec2}, we present some 2- and 3-dissections that will be used in the subsequent sections. In Section \ref{sec3}, we prove Theorem \ref{thm-lin-pdot} whereas Section \ref{sec4} is devoted to proving the congruences \eqref{pdt-lin1} and \eqref{pdt-lin2}. In Section \ref{sec5}, we present the proofs of our new congruences in Theorem \ref{thm-mod2}.

\section{\textbf{Some 2- and 3-dissections}}\label{sec2}
In this section, we present some useful 2- and 3-dissections.

\begin{lemma} We have
\begin{align}
\label{1byf1sq}\dfrac{1}{f_1^2}&=\dfrac{f_8^5}{f_2^5f_{16}^2}+2q\dfrac{f_4^2f_{16}^2}{f_2^5f_8},\\
\label{f1-2}f_1^2&=\dfrac{f_2f_8^5}{f_4^2f_{16}^2}-2q\dfrac{f_2f_{16}^2}{f_8},\\
\label{1byf1four}\dfrac{1}{f_1^4}&=\dfrac{f_4^{14}}{f_2^{14}f_8^4}+4q\dfrac{f_4^2f_8^4}{f_2^{10}},\\
\label{f1-4}f_1^4&=\dfrac{f_4^{10}}{f_2^2f_8^4}-4q\dfrac{f_2^2f_8^4}{f_4^2},\\
\label{f1f3}f_1f_3&=\dfrac{f_2f_8^2f_{12}^4}{f_4^2f_6f_{24}^2}-q\dfrac{f_4^4f_6f_{24}^2}{f_2f_8^2f_{12}^2},\\
\label{1byf1f3}\dfrac{1}{f_1f_3}&=\dfrac{f_8^2f_{12}^5}{f_2^2f_4f_6^4f_{24}^2}+q\dfrac{f_4^5f_{24}^2}{f_2^4f_6^2f_8^2f_{12}},\\
\label{f1-3byf3}\dfrac{f_1^3}{f_3}&=\dfrac{f_4^3}{f_{12}}-3q\dfrac{f_2^2f_{12}^3}{f_4f_6^2},\\
\label{f3byf1-3}\dfrac{f_3}{f_1^3}&=\dfrac{f_4^6f_6^3}{f_2^9f_{12}^2}+3q\dfrac{f_4^2f_6f_{12}^2}{f_2^7},\\
\label{f3-3byf1}\dfrac{f_3^3}{f_1}&=\dfrac{f_4^3f_6^2}{f_2^2f_{12}}+q\dfrac{f_{12}^3}{f_4},\\
\label{f1-2byf3-2}\dfrac{f_1^2}{f_3^2}&=\dfrac{f_2f_4^2f_{12}^4}{f_6^5f_8f_{24}}-2q\dfrac{f_2^2f_8f_{12}f_{24}}{f_4f_6^4},\\
\label{f3-2byf1-2}\dfrac{f_3^2}{f_1^2}&=\dfrac{f_4^4f_6f_{12}^2}{f_2^5f_8f_{24}}+2q\dfrac{f_4f_6^2f_8f_{24}}{f_2^4f_{12}}.
\end{align}
\end{lemma}
\begin{proof}
  Identities \eqref{1byf1sq} and \eqref{1byf1four} are the 2-dissections of $\varphi(q)$ and $\varphi(q^2)$ (see \cite[Eqs. (1.9.4) and (1.10.1)]{hirschhorn}), where
   $$\varphi(q):=\sum_{n=-\infty}^\infty q^{n^2}=\dfrac{f_2^5}{f_1^2f_4^2}.$$
   Replacing $q$ by $-q$ in \eqref{1byf1sq} and \eqref{1byf1four}, and then using \begin{align}\label{-q-q}(-q;-q)_\infty=\dfrac{f_2^2}{f_1f_4},\end{align} we readily arrive at \eqref{f1-2} and \eqref{f1-4}, respectively.  Identities  \eqref{f1f3}, \eqref{1byf1f3}, \eqref{f1-3byf3}, \eqref{f3-3byf1}, \eqref{f1-2byf3-2}, and  \eqref{f3-2byf1-2} are Eqs. (30.12.1), (30.12.3), (22.1.13), (22.1.14), (30.10.2), and (30.10.4), respectively, in  \cite{hirschhorn}. Finally, \eqref{f3byf1-3} follows from \eqref{f1-3byf3} by replacing $q$ by $-q$ and then using \eqref{-q-q}.
   \end{proof}

\begin{lemma}\label{lemma6} We have
\begin{align}
\label{f1sq-byf2}\dfrac{f_1^2}{f_2}&=\dfrac{f_9^2}{f_{18}}-2q\dfrac{f_3f_{18}^2}{f_6f_9},\\
\label{f2-f1sq}\dfrac{f_2}{f_1^2}&=\dfrac{f_6^4f_9^6}{f_3^8f_{18}^3}+2q\dfrac{f_6^3f_9^3}{f_3^7}+4q^2\dfrac{f_6^2f_{18}^3}{f_3^6},\\
\label{f1-3}f_1^3&=f_3a(q^3)-3qf_9^3,\\
\label{1byf1-3}\dfrac{1}{f_1^3}&=a^2(q^3)\dfrac{f_9^3}{f_3^{10}}+3qa(q^3)\dfrac{f_9^6}{f_3^{11}}+9q^2\dfrac{f_9^9}{f_3^{12}},\\
\label{1byf1f2}\dfrac{1}{f_1f_2}&=a(q^6)\dfrac{f_9^3}{f_3^{4}f_6^3}+qa(q^3)\dfrac{f_{18}^3}{f_3^{3}f_6^4}+3q^2\dfrac{f_9^3f_{18}^3}{f_3^{4}f_6^4},
\end{align}
where
\begin{align*}
a(q):=\sum_{m,n={-\infty}}^\infty q^{m^2+mn+n^2}=1+6\sum_{n=0}^\infty \left(\dfrac{q^{3n+1}}{1-q^{3n+1}}-\dfrac{q^{3n+2}}{1-q^{3n+2}}\right).
\end{align*}
\end{lemma}

\begin{proof}The first identity is equivalent to the 3-dissection of $\varphi(-q)$ (see \cite[Eq. (14.3.2)]{hirschhorn}). The second can be obtained from the first by
replacing $q$ with $\omega q$ and $\omega^2q$ and then multiplying
the two results, where $\omega$ is a primitive cube root of unity. Identities \eqref{f1-3}, \eqref{1byf1-3} and \eqref{1byf1f2} are in \cite[Eqs. (21.3.1), (39.2.8) and (22.9.4)]{hirschhorn}.
\end{proof}

We also recall the following useful results from \cite[Eqs. (22.1.12), (22.11.8) and (22.11.9)]{hirschhorn}, where the first is a 2-dissection of $a(q)$:
\begin{align}
\label{aq-q4}a(q)&=a(q^4)+6q\dfrac{f_4^2f_{12}^2}{f_2f_6},\\
\label{a-2q2}a(q)&+2a(q^2)=3\dfrac{f_2f_{3}^6}{f_1^2f_6^3},\\
\label{a-q2}a(q)&+a(q^2)=2\dfrac{f_2^6f_3}{f_1^3f_6^2}.
\end{align}

We end this section by noting the following congruences which can be easily established:
\begin{align*}
a(q)&\equiv 1~(\textup{mod}~2),\\
a^2(q)&\equiv 1~(\textup{mod}~4),\\
f_1^2&\equiv f_{2}~(\textup{mod}~2),\\
f_1^4&\equiv f_{2}^2~(\textup{mod}~4),\\
f_1^8&\equiv f_{2}^4~(\textup{mod}~8).
\end{align*}

We will frequently use the identities and congruences of this section in the subsequent sections, sometimes without referring to these.

\section{\textbf{Proof of Theorem \ref{thm-lin-pdot}}}\label{sec3}
We have
\begin{align}
\label{rav1}\sum_{n=0}^{\infty}\textup{PDO}_\textup{t}(n)q^n&=q\dfrac{f_2f_{12}^2}{f_6}\cdot\dfrac{f_3^2}{f_1^2}\notag\\
&=q\dfrac{f_2f_{12}^2}{f_6}\left(\dfrac{f_4^4f_6f_{12}^2}{f_2^5f_8f_{24}}+2q\dfrac{f_4f_6^2f_8f_{24}}{f_2^4f_{12}}\right),
\end{align}
from which we extract
\begin{align*}
\sum_{n=0}^{\infty}\textup{PDO}_\textup{t}(2n)q^n&=2q{f_2f_4f_6f_{12}}\cdot\dfrac{f_3}{f_1^3}\notag\\
&=2q{f_2f_4f_6f_{12}}\left(\dfrac{f_4^6f_6^3}{f_2^9f_{12}^2}+3q\dfrac{f_4^2f_6f_{12}^2}{f_2^7}\right).
\end{align*}
From the above, we extract
\begin{align*}
&\sum_{n=0}^{\infty}\textup{PDO}_\textup{t}(4n+2)q^n\\
&=2\dfrac{f_2^7}{f_6}\cdot\dfrac{1}{f_1^8}\cdot f_3^4\notag\\
&=2\dfrac{f_2^7}{f_6}\left(\dfrac{f_4^{14}}{f_2^{14}f_8^4}+4q\dfrac{f_4^2f_8^4}{f_2^{10}}\right)^2\left(\dfrac{f_{12}^{10}}{f_6^{2}f_{24}^4}-4q^3
\dfrac{f_6^2f_{24}^4}{f_{12}^{2}}\right)\\
&=2\dfrac{f_4^{28}f_{12}^{10}}{f_2^{21}f_6^3f_8^8f_{24}^4}+16q\dfrac{f_4^{16}f_{12}^{10}}{f_2^{17}f_6^3f_{24}^4}
+32q^2\dfrac{f_4^{4}f_8^8f_{12}^{10}}{f_2^{13}f_6^3f_{24}^4}\\
&\quad -8q^3\dfrac{f_4^{28}f_6f_{24}^{4}}{f_2^{21}f_8^8f_{12}^2}-64q^4\dfrac{f_4^{16}f_6f_{24}^{4}}{f_2^{17}f_{12}^2}-128q^5
\dfrac{f_4^{4}f_6f_8^8f_{24}^{4}}{f_2^{13}f_{12}^2},
\end{align*}
from which we extract
\begin{align*}
\sum_{n=0}^{\infty}\textup{PDO}_\textup{t}(8n+6)q^n&=16\dfrac{f_2^{16}f_{6}^{10}}{f_1^{17}f_3^3f_{12}^4}
-8q\dfrac{f_2^{28}f_3f_{12}^{4}}{f_1^{21}f_4^8f_{6}^2}-128q^2
\dfrac{f_2^{4}f_3f_4^8f_{12}^{4}}{f_1^{13}f_{6}^2},
\end{align*}
which is \eqref{thm-pdot-lin1}.

Next, from \eqref{rav1} we also extract
\begin{align*}
\sum_{n=0}^{\infty}\textup{PDO}_\textup{t}(2n+1)q^n&=\dfrac{f_2^4f_6^4}{f_4f_{12}}\cdot\dfrac{1}{f_1^4}\\
&=\dfrac{f_2^4f_6^4}{f_4f_{12}}\cdot\left(\dfrac{f_4^{14}}{f_2^{14}f_8^4}+4q\dfrac{f_4^2f_8^4}{f_2^{10}}\right),
\end{align*}
from which we have
\begin{align*}
&\sum_{n=0}^{\infty}\textup{PDO}_\textup{t}(4n+3)q^n\\
&=4\dfrac{f_2f_4^4}{f_6}\cdot\dfrac{1}{f_1^2}\cdot\dfrac{f_3^4}{f_1^4}\\
&=4\dfrac{f_2f_4^4}{f_6}\left(\dfrac{f_8^5}{f_2^5f_{16}^2}+2q\dfrac{f_4^2f_{16}^2}{f_2^5f_8}\right)
\left(\dfrac{f_4^4f_6f_{12}^2}{f_2^5f_8f_{24}}+2q\dfrac{f_4f_6^2f_8f_{24}}{f_2^4f_{12}}\right)^2\\
&=4\dfrac{f_4^{12}f_6f_8^3f_{12}^4}{f_2^{14}f_{16}^2f_{24}^2}+16q\dfrac{f_4^9f_6^2f_8^5f_{12}}{f_2^{13}f_{16}^2}
+8q\dfrac{f_4^{14}f_6f_{12}^4f_{16}^2}{f_2^{14}f_8^3f_{24}^2}\\
&\quad+16q^2\dfrac{f_4^{6}f_6^3f_8^7f_{24}^2}{f_2^{12}f_{12}^2f_{16}^2}+32q^2\dfrac{f_4^{11}f_6^2f_{12}f_{16}^2}{f_2^{13}f_{8}}
+32q^3\dfrac{f_4^{8}f_6^3f_8f_{16}^2f_{24}^2}{f_2^{12}f_{12}^2},
\end{align*}
from which we extract
\begin{align*}
\sum_{n=0}^{\infty}\textup{PDO}_\textup{t}(8n+7)q^n&=8\dfrac{f_2^{14}f_3f_{6}^4f_{8}^2}{f_1^{14}f_4^3f_{12}^2}+16\dfrac{f_2^9f_3^2f_4^5f_{6}}{f_1^{13}f_{8}^2}
+32q\dfrac{f_2^{8}f_3^3f_4f_{8}^2f_{12}^2}{f_1^{12}f_{6}^2},
\end{align*}
which is \eqref{thm-pdot-lin2}.

\section{\textbf{Proofs of \eqref{pdt-lin1} and \eqref{pdt-lin2}}}\label{sec4}
We have
\begin{align}\label{pdt-n}
2\sum_{n=0}^{\infty}\textup{PD}_\textup{t}(n)q^n&=\dfrac{f_3^5}{f_6^2}\cdot\dfrac{1}{f_1^3}-\dfrac{f_6}{f_3}\cdot\dfrac{1}{f_1f_2}\notag\\
&=\dfrac{f_3^5}{f_6^2}\left(a^2(q^3)\dfrac{f_9^3}{f_3^{10}}+3qa(q^3)\dfrac{f_9^6}{f_3^{11}}+9q^2\dfrac{f_9^9}{f_3^{12}}\right)\notag\\
&\quad-\dfrac{f_6}{f_3}\left(a(q^6)\dfrac{f_9^3}{f_3^{4}f_6^3}+qa(q^3)\dfrac{f_{18}^3}{f_3^{3}f_6^4}+3q^2\dfrac{f_9^3f_{18}^3}{f_3^{4}f_6^4}\right),
\end{align}
from which we extract
\begin{align*}
2\sum_{n=0}^{\infty}\textup{PD}_\textup{t}(3n+1)q^n&=3a(q)\dfrac{f_3^6}{f_1^6f_2^2}-a(q)\dfrac{f_6^3}{f_1^4f_2^3}\\
&=\left(3\dfrac{f_3^6}{f_1^6f_2^2}-\dfrac{f_6^3}{f_1^4f_2^3}\right)a(q)\\
&=\dfrac{f_6^3}{f_1^4f_2^3}\left(3\dfrac{f_2f_3^6}{f_1^2f_6^3}-1\right)a(q)\\
&=\dfrac{f_6^3}{f_1^4f_2^3}\left(a(q)+2a(q^2)-1\right)a(q)\notag\\
&=\dfrac{f_6^3}{f_1^4f_2^3}\left(a(q)+a(-q)+2a(q^2)-1-a(-q)\right)a(q)\\
&=\dfrac{f_6^3}{f_1^4f_2^3}\left(2a(q^4)+2a(q^2)-1-a(-q)\right)a(q)\\
&=\dfrac{f_6^3}{f_1^4f_2^3}\left(4\dfrac{f_4^6f_6}{f_2^3f_{12}^2}-1-a(-q)\right)a(q)\\
&=\dfrac{f_6^3}{f_1^4f_2^3}\left(\left(4\dfrac{f_4^6f_6}{f_2^3f_{12}^2}-1\right)\left(a(q^4)+6q\dfrac{f_4^2f_{12}^2}{f_2f_6}\right)-a(-q)a(q)\right)\\
&=\dfrac{f_6^3}{f_1^4f_2^3}\Bigg(\left(4\dfrac{f_4^6f_6}{f_2^3f_{12}^2}-1\right)\left(a(q^4)+6q\dfrac{f_4^2f_{12}^2}{f_2f_6}\right)\\
&\quad-\left(a^2(q^4)-36q^2\dfrac{f_4^4f_{12}^4}{f_2^2f_6^2}\right)\Bigg)\\
&=\dfrac{f_6^3}{f_2^3}\Bigg(\left(4\dfrac{f_4^6f_6}{f_2^3f_{12}^2}a(q^4)-a(q^4)-a^2(q^4)+36q^2\dfrac{f_4^4f_{12}^4}{f_2^2f_6^2}\right)\\
&\quad+q\left(24\dfrac{f_4^8}{f_2^4}-6\dfrac{f_4^2f_{12}^2}{f_2f_6}\right)\Bigg)\left(\dfrac{f_4^{14}}{f_2^{14}f_8^4}+4q\dfrac{f_4^2f_8^4}{f_2^{10}}\right).
\end{align*}
Extracting the terms involving $q^{2n+1}$ from both sides of the above and then dividing by 2, we find that
\begin{align*}
\sum_{n=0}^{\infty}\textup{PD}_\textup{t}(6n+4)q^n&=\dfrac{f_3^3}{f_1^3}\Bigg(2\dfrac{f_2^2f_4^4}{f_1^{10}}
\left(4\dfrac{f_2^6f_3}{f_1^3f_{6}^2}a(q^2)-a(q^2)-a^2(q^2)+36q\dfrac{f_2^4f_{6}^4}{f_1^2f_3^2}\right)\\
&\quad+\dfrac{f_2^{14}}{f_1^{14}f_4^4}\left(12\dfrac{f_2^8}{f_1^4}-3\dfrac{f_2^2f_{6}^2}{f_1f_3}\right)\Bigg).
\end{align*}
Taking congruences modulo 8, we have
\begin{align}\label{pd-main}
&\sum_{n=0}^{\infty}\textup{PD}_\textup{t}(6n+4)q^n\notag\\
&\equiv\dfrac{f_3^3}{f_1^3}\Bigg(6\dfrac{f_2^2f_4^4}{f_1^{10}}\left(a(q^2)+a^2(q^2)\right)+4\dfrac{f_2^{22}}{f_1^{18}f_4^4}+5\dfrac{f_2^{16}f_{6}^2}{f_1^{15}f_3f_4^4}
\Bigg)\notag\\
&\equiv\dfrac{f_3^3}{f_1^3}\Bigg(6\dfrac{f_2^2f_4^4}{f_1^{10}}a(q^2)+6\dfrac{f_2^2f_4^4}{f_1^{10}}+4\dfrac{f_2^{22}}{f_1^{18}f_4^4}+5\dfrac{f_2^{16}f_{6}^2}{f_1^{15}f_3f_4^4}
\Bigg)\notag\\
&\equiv\left(6f_4^2f_6^2a(q^2)+10f_4^2f_6^2\right)\cdot\dfrac{1}{f_1f_3}+5f_{6}^2\cdot\dfrac{f_3^2}{f_1^2}\notag\\
&\equiv\left(6f_4^2f_6^2a(q^2)+10f_4^2f_6^2\right)\left(\dfrac{f_8^2f_{12}^5}{f_2^2f_4f_6^4f_{24}^2}+q\dfrac{f_4^5f_{24}^2}{f_2^4f_6^2f_8^2f_{12}}\right)\notag\\
&\quad+5f_{6}^2\left(\dfrac{f_4^4f_6f_{12}^2}{f_2^5f_8f_{24}}+2q\dfrac{f_4f_6^2f_8f_{24}}{f_2^4f_{12}}\right).
\end{align}
We extract
\begin{align*}
&\sum_{n=0}^{\infty}\textup{PD}_\textup{t}(12n+4)q^n\notag\\
&\equiv\left(6f_2^2f_3^2a(q)+10f_2^2f_3^2\right)\dfrac{f_4^2f_{6}^5}{f_1^2f_2f_3^4f_{12}^2}+5\dfrac{f_2^4f_3^3f_{6}^2}{f_1^5f_4f_{12}}\\
&\equiv6\dfrac{f_2f_4^2}{f_6}\cdot a(q)\cdot \dfrac{f_3^2}{f_1^2}+10f_{4}^2+5\dfrac{f_6^2}{f_4f_{12}}\cdot f_1^3f_3^3\\
&\equiv6\dfrac{f_2f_4^2}{f_6} \left(a(q^4)+6q\dfrac{f_4^2f_{12}^2}{f_2f_6}\right) \left(\dfrac{f_4^4f_6f_{12}^2}{f_2^5f_8f_{24}}+2q\dfrac{f_4f_6^2f_8f_{24}}{f_2^4f_{12}}\right)\\
&\quad+10f_{4}^2+5\dfrac{f_6^2}{f_4f_{12}}\left(\dfrac{f_2f_8^2f_{12}^4}{f_4^2f_6f_{24}^2}-q\dfrac{f_4^4f_6f_{24}^2}{f_2f_8^2f_{12}^2}\right)^3\\
&\equiv6\Bigg(a(q^4)\dfrac{f_4^6f_{12}^2}{f_2^4f_8f_{24}}+12q^2\dfrac{f_4^5f_8f_{12}f_{24}}{f_2^4}
+q\Bigg(2a(q^4)\dfrac{f_4^3f_6f_8f_{24}}{f_2^3f_{12}}\\
&\quad+6\dfrac{f_4^8f_{12}^4}{f_2^5f_6f_8f_{24}}\Bigg)\Bigg)+10f_{4}^2+5\Bigg(\dfrac{f_2^3f_8^6f_{12}^{11}}{f_4^7f_6f_{24}^6}-3q\dfrac{f_2f_6f_8^2f_{12}^5}{f_4f_{24}^2}\\
&\quad+3q^2\dfrac{f_4^5f_6^3f_{24}^2}{f_2f_8^2f_{12}}-q^3\dfrac{f_4^{11}f_6^5f_{24}^6}{f_2^3f_8^6f_{12}^7}\Bigg)~(\textup{mod}~8),
\end{align*}from which we extract
\begin{align*}
&\sum_{n=0}^{\infty}\textup{PD}_\textup{t}(24n+4)q^n\\
&\equiv6a(q^2)\dfrac{f_2^6f_{6}^2}{f_1^4f_4f_{12}}+10f_{2}^2+5\Bigg(\dfrac{f_1^3f_4^6f_{6}^{11}}{f_2^7f_3f_{12}^6}
+3q\dfrac{f_2^5f_3^3f_{12}^2}{f_1f_4^2f_{6}}\Bigg)\\
&\equiv6a(q^2)\dfrac{f_4f_{6}^2}{f_{12}}+10f_{2}^2+5\left(\dfrac{f_2f_4^2f_6^3}{f_{12}^3}\cdot \dfrac{f_1^3}{f_3}+3q \dfrac{f_2^5f_{12}^2}{f_4^2f_6}\cdot \dfrac{f_3^3}{f_1}\right)\\
&\equiv6a(q^2)\dfrac{f_4f_{6}^2}{f_{12}}+10f_{2}^2+5\Bigg(\dfrac{f_2f_4^2f_6^3}{f_{12}^2}\Bigg(\dfrac{f_4^3}{f_{12}}-3q\dfrac{f_2^2f_{12}^3}{f_4f_6^2}\Bigg)\\
&\quad +3q \dfrac{f_2^5f_{12}^2}{f_4^2f_6} \Bigg(\dfrac{f_4^3f_6^2}{f_2^2f_{12}}+q\dfrac{f_{12}^3}{f_4}\Bigg)\Bigg)\\
&\equiv6a(q^2)\dfrac{f_4f_{6}^2}{f_{12}}+10f_{2}^2+5\Bigg(\dfrac{f_2f_4^5f_6^3}{f_{12}^3} +3q^2 \dfrac{f_2^5f_{12}^5}{f_4^3f_6}\Bigg)
~(\textup{mod}~8).
\end{align*}
Therefore,
\begin{align*}
\textup{PD}_\textup{t}(48n+28)&\equiv0~(\textup{mod}~8).
\end{align*}
which is \eqref{pdt-lin1}.

Now, from \eqref{pd-main}, we also extract
\begin{align*}
&\sum_{n=0}^{\infty}\textup{PD}_\textup{t}(12n+10)q^n\notag\\
&\equiv\left(6f_2^2f_3^2a(q)+10f_2^2f_3^2\right)\dfrac{f_2^5f_{12}^2}{f_1^4f_3^2f_4^2f_{6}}
+10\dfrac{f_2f_3^4f_4f_{12}}{f_1^4f_{6}}\\
&\equiv6\dfrac{f_2f_{12}^2}{f_6}\cdot a(q)+10\dfrac{f_2f_{12}^2}{f_6}+10\dfrac{f_2f_6f_{12}}{f_2}\\
&\equiv6\dfrac{f_2f_{12}^2}{f_6}\left(a(q^4)+6q\dfrac{f_4^2f_{12}^2}{f_2f_6}\right)+10\dfrac{f_2f_{12}^2}{f_6}+10\dfrac{f_2f_6f_{12}}{f_2}~(\textup{mod}~8),
\end{align*}
from which we extract
\begin{align*}
\sum_{n=0}^{\infty}\textup{PD}_\textup{t}(24n+22)q^n&\equiv36\dfrac{f_2^2f_{6}^4}{f_3^2}\equiv36f_2^2f_{6}^3~(\textup{mod}~8),
\end{align*}
Thus,
\begin{align*}
\textup{PD}_\textup{t}(48n+46)&\equiv0~(\textup{mod}~8),
\end{align*}
which is \eqref{pdt-lin2}.

\section{\textbf{Proof of Theorem \ref{thm-mod2}}}\label{sec5}

From \eqref{pdt-n} we extract
\begin{align}\label{pdt-3n}
2\sum_{n=0}^{\infty}\textup{PD}_\textup{t}(3n)q^n&=\dfrac{f_3^3}{f_1^5f_2^2}a^2(q)-\dfrac{f_3^3}{f_1^{5}f_2^2}a(q^2)\notag\\
&\equiv\dfrac{1}{f_4^2}\cdot\dfrac{f_3^3}{f_1}-\dfrac{a(q^2)}{f_4^2}\cdot\dfrac{f_3^3}{f_1}\notag\\
&\equiv\left(\dfrac{1}{f_4^2}-\dfrac{a(q^2)}{f_4^2}\right)\left(\dfrac{f_4^3f_6^2}{f_2^2f_{12}}+q\dfrac{f_{12}^3}{f_4}\right)~(\textup{mod}~4),
\end{align}
from which we extract
\begin{align}\label{pdt-6n-2}
&2\sum_{n=0}^{\infty}\textup{PD}_\textup{t}(6n)q^n\notag\\
&\equiv\left(\dfrac{1}{f_2^2}-\dfrac{a(q)}{f_2^2}\right)\dfrac{f_2^3f_3^2}{f_1^2f_{6}}\notag\\
&\equiv\dfrac{f_2}{f_6}\cdot \dfrac{f_3^2}{f_1^2}-\dfrac{f_2}{f_6}\cdot \dfrac{f_3^2}{f_1^2}\cdot a(q)\notag\\
&\equiv\dfrac{f_2}{f_6}\left(\dfrac{f_4^4f_6f_{12}^2}{f_2^5f_8f_{24}}+2q\dfrac{f_4f_6^2f_8f_{24}}{f_2^4f_{12}}\right)\notag\\
&\quad-\dfrac{f_2}{f_6}\left(\dfrac{f_4^4f_6f_{12}^2}{f_2^5f_8f_{24}}+2q\dfrac{f_4f_6^2f_8f_{24}}{f_2^4f_{12}}\right)
\left(a(q^4)+6q\dfrac{f_4^2f_{12}^2}{f_2f_6}\right)~(\textup{mod}~4).
\end{align}
Therefore,
\begin{align*}
2\sum_{n=0}^{\infty}\textup{PD}_\textup{t}(12n)q^n&\equiv\dfrac{f_2^4f_6^2}{f_1^4f_4f_{12}}-\dfrac{f_1}{f_3}\dfrac{f_2^4f_3f_{6}^2}{f_1^5f_4f_{12}}~a(q^2)\notag\\
&\equiv\dfrac{f_2^2f_6^2}{f_4f_{12}}-\dfrac{f_2^2f_{6}^2}{f_4f_{12}}~a(q^2)~(\textup{mod}~4),
\end{align*}
from which we readily arrive at
\begin{align*}
\textup{PD}_\textup{t}(24n+12)&\equiv0~(\textup{mod}~2),
\end{align*}
which is \eqref{24-12}.

Next, extracting the terms involving $q^{2n+1}$ from both sides of \eqref{pdt-6n-2}, and then dividing by 2, we have
\begin{align*}
\sum_{n=0}^{\infty}\textup{PD}_\textup{t}(12n+6)q^n&\equiv\dfrac{f_2f_3f_4f_{12}}{f_1^3f_{6}}
-a(q^2)\dfrac{f_2f_3f_4f_{12}}{f_1^3f_{6}}-3\dfrac{f_2^6f_6^4}{f_1^5f_3f_4f_{12}}\\
&\equiv\dfrac{f_2f_4f_{12}}{f_{6}}\cdot\dfrac{f_3}{f_1^3}
\notag\\
&\equiv \dfrac{f_2f_4f_{12}}{f_6}\left(\dfrac{f_4^6f_6^3}{f_2^9f_{12}^2}+3q\dfrac{f_4^2f_6f_{12}^2}{f_2^7}\right)~(\textup{mod}~2),
\end{align*}
from which we extract
\begin{align*}
\sum_{n=0}^{\infty}\textup{PD}_\textup{t}(24n+6)q^n&\equiv \dfrac{f_1f_2f_{6}}{f_3}\cdot\dfrac{f_2^6f_3^3}{f_1^9f_{6}^2}\equiv
f_2^3~(\textup{mod}~2),
\end{align*}
from which we further extract
\begin{align*}
\textup{PD}_\textup{t}(48n+30)&\equiv 0~(\textup{mod}~2),
\end{align*}
which is \eqref{48n30}, and
\begin{align*}
\sum_{n=0}^{\infty}\textup{PD}_\textup{t}(48n+6)q^n&\equiv f_1^3\equiv f_3a(q^3)-3qf_9^3~(\textup{mod}~2),
\end{align*}
which implies
\begin{align*}
\textup{PD}_\textup{t}(144n+102)&\equiv 0~(\textup{mod}~2),
\end{align*}which is \eqref{eq133}.

Now, from \eqref{pdt-3n} we extract
\begin{align*}
2\sum_{n=0}^{\infty}\textup{PD}_\textup{t}(6n+3)q^n
&\equiv\left(\dfrac{1}{f_2^2}-\dfrac{a(q)}{f_2^2}\right)\dfrac{f_{6}^3}{f_2}\\
&\equiv \dfrac{f_{6}^3}{f_2^3}-\dfrac{f_{6}^3}{f_2^3}\left(a(q^4)+6q\dfrac{f_4^2f_{12}^2}{f_2f_6}\right)~(\textup{mod}~4),
\end{align*}
from which we extract
\begin{align}\label{12n9-2}
\sum_{n=0}^{\infty}\textup{PD}_\textup{t}(12n+9)q^n&\equiv \dfrac{f_2^2f_3^2f_{6}^2}{f_1^4}\equiv f_6^3~(\textup{mod}~2).
\end{align}
This implies
\begin{align*}
\textup{PD}_\textup{t}(24n+21)&\equiv 0~(\textup{mod}~2),
\end{align*}
which is \eqref{24-21}. Furthermore,
\begin{align*}
\textup{PD}_\textup{t}(36n+21)\equiv \textup{PD}_\textup{t}(36n+33)\equiv 0~(\textup{mod}~2),
\end{align*}
which are weaker versions of \eqref{36n21} and \eqref{36n33}.

From \eqref{12n9-2} we also extract
\begin{align}
\label{72-9}\sum_{n=0}^{\infty}\textup{PD}_\textup{t}(72n+9)q^n\equiv f_1^3=f_3a(q^3)-3qf_9^3~(\textup{mod}~2),
\end{align}
from which we further extract
\begin{align*}
\textup{PD}_\textup{t}(216n+153)\equiv 0~(\textup{mod}~2),
\end{align*}
which is \eqref{216-153}.

From \eqref{pdt-n}, we extract
\begin{align*}
2\sum_{n=0}^{\infty}\textup{PD}_\textup{t}(3n)q^n&= \dfrac{f_3^3}{f_1^5f_2^2}a^2(q)-\dfrac{f_3^3}{f_1^5f_2^2}a(q^2)\notag\\
&= \dfrac{f_3^3}{f_1^5f_2^2}\left(2(a(q)+a(q^2))-(a(q)+2a(q^2))\right)^2\notag\\
&\quad-\dfrac{f_3^3}{f_1^5f_2^2}\left((a(q)+2a(q^2))-(a(q)+a(q^2))\right)\notag\\
&= \dfrac{f_3^3}{f_1^5f_2^2}\left(4\dfrac{f_2^6f_3}{f_1^3f_6^2}-3\dfrac{f_2f_3^6}{f_1^2f_6^3}\right)^2-\dfrac{f_3^3}{f_1^5f_2^2}\left(3\dfrac{f_2f_3^6}{f_1^2f_6^3}-2\dfrac{f_2^6f_3}{f_1^3f_6^2}\right).
\end{align*}
Taking congruences modulo 8, we have
\begin{align}\label{3n-mod8}
2\sum_{n=0}^{\infty}\textup{PD}_\textup{t}(3n)q^n&\equiv \dfrac{f_6^2}{f_2^4}\cdot\dfrac{1}{f_1f_3}+5\dfrac{f_6}{f_2^5}\cdot f_1f_3 +2\notag\\
&\equiv \dfrac{f_6^2}{f_2^4}\left(\dfrac{f_8^2f_{12}^5}{f_2^2f_4f_6^4f_{24}^2}+q\dfrac{f_4^5f_{24}^2}{f_2^4f_6^2f_8^2f_{12}}\right)\notag\\
&\quad+5\dfrac{f_6}{f_2^5}\left(\dfrac{f_2f_8^2f_{12}^4}{f_4^2f_6f_{24}^2}-q\dfrac{f_4^4f_6f_{24}^2}{f_2f_8^2f_{12}^2}\right)+2,
\end{align}
from which we extract
\begin{align*}
&2\sum_{n=0}^{\infty}\textup{PD}_\textup{t}(6n+3)q^n\\
&\equiv \dfrac{f_2^5f_{12}^2}{f_1^8f_4^2f_{6}}+3\dfrac{f_2^4f_3^2f_{12}^2}{f_1^6f_4^2f_{6}^2}\\
&\equiv \dfrac{f_2f_{12}^2}{f_4^2f_{6}}+3\dfrac{f_{12}^2}{f_4^2f_6^2}\cdot f_1^2f_3^2\\
&\equiv \dfrac{f_2f_{12}^2}{f_4^2f_{6}}+3\dfrac{f_{12}^2}{f_4^2f_6^2}\cdot \left(\dfrac{f_2f_8^2f_{12}^4}{f_4^2f_6f_{24}^2}-q\dfrac{f_4^4f_6f_{24}^2}{f_2f_8^2f_{12}^2}\right)^2\\
&\equiv \dfrac{f_2f_{12}^2}{f_4^2f_{6}}+3\dfrac{f_{12}^2}{f_4^2f_6^2}\cdot \left(\dfrac{f_2^2f_8^4f_{12}^8}{f_4^4f_6^2f_{24}^4}-2qf_4^2f_{12}^2+q^2\dfrac{f_4^8f_6^2f_{24}^4}{f_2^2f_8^4f_{12}^4}\right)~(\textup{mod}~8).
\end{align*}
Extracting the terms involving $q^{2n+1}$ from both sides, and then dividing by 2, we have
\begin{align*}
\sum_{n=0}^{\infty}\textup{PD}_\textup{t}(12n+9)q^n&\equiv \dfrac{f_6^4}{f_3^2}~(\textup{mod}~4),
\end{align*}
from which we extract
\begin{align*}
\textup{PD}_\textup{t}(36n+21)&\equiv 0~(\textup{mod}~4)\\\intertext{and}
\textup{PD}_\textup{t}(36n+22)&\equiv 0~(\textup{mod}~4),
\end{align*}
which are \eqref{36n21} and \eqref{36n33}, respectively.

Now, from \eqref{3n-mod8}, we extract
\begin{align}\label{pdt-6n}2\sum_{n=0}^{\infty}\textup{PD}_\textup{t}(6n)q^n&\equiv \dfrac{f_4^2f_{6}^5}{f_1^6f_2f_3^2f_{12}^2}+5\dfrac{f_4^2f_{6}^4}{f_1^4f_2^2f_{12}^2}+2\notag\\
&\equiv \dfrac{f_4^2f_6^5}{f_2^5f_{12}^2}\cdot\dfrac{f_1^2}{f_3^2}-\dfrac{f_4^2f_6^4}{f_2^2f_{12}^2}\cdot\dfrac{1}{f_1^4}+2\notag\\
&\equiv \dfrac{f_4^2f_6^5}{f_2^5f_{12}^2}\left(\dfrac{f_2f_4^2f_{12}^4}{f_6^5f_8f_{24}}-2q\dfrac{f_2^2f_8f_{12}f_{24}}{f_4f_6^4}\right)\notag\\
&\quad-\dfrac{f_4^2f_6^4}{f_2^2f_{12}^2}\left(\dfrac{f_4^{14}}{f_2^{14}f_8^4}+4q\dfrac{f_4^2f_8^4}{f_2^{10}}\right)+2~(\textup{mod}~8),
\end{align}
which yields
\begin{align}
\label{12n}2\sum_{n=0}^{\infty}\textup{PD}_\textup{t}(12n)q^n&\equiv \dfrac{f_2^4f_6^2}{f_4f_{12}}\cdot\dfrac{1}{f_1^4}-f_6^2\cdot \dfrac{1}{f_3^4}+2\notag\\
&\equiv \dfrac{f_2^4f_6^2}{f_4f_{12}}\left(\dfrac{f_4^{14}}{f_2^{14}f_8^4}+4q\dfrac{f_4^2f_8^4}{f_2^{10}}\right)\notag\\
&\quad-f_6^2\left(\dfrac{f_{12}^{14}}{f_6^{14}f_{24}^4}+4q^3\dfrac{f_{12}^2f_{24}^4}{f_6^{10}}\right)+2~(\textup{mod}~8),
\end{align}
from which we extract
\begin{align}
\label{c11}2\sum_{n=0}^{\infty}\textup{PD}_\textup{t}(24n)q^n&\equiv \dfrac{f_2}{f_6}\cdot\dfrac{f_3^2}{f_1^2}-f_6^2\cdot\dfrac{1}{f_3^4}+2\notag\\
&\equiv \dfrac{f_2}{f_6}\left(\dfrac{f_4^4f_6f_{12}^2}{f_2^5f_8f_{24}}+2q\dfrac{f_4f_6^2f_8f_{24}}{f_2^4f_{12}}\right)\notag\\
&\quad-f_6^2\left(\dfrac{f_{12}^{14}}{f_6^{14}f_{24}^4}+4q^3\dfrac{f_{12}^2f_{24}^4}{f_6^{10}}\right)+2~(\textup{mod}~8).
\end{align}
We extract
\begin{align}
\label{48n}2\sum_{n=0}^{\infty}\textup{PD}_\textup{t}(48n)q^n&\equiv \dfrac{f_2^4f_6^2}{f_4f_{12}}\cdot\dfrac{1}{f_1^4}-\dfrac{f_6^2}{f_3^4}+2~(\textup{mod}~8).
\end{align}
From \eqref{12n} and \eqref{48n}, we arrive at
\begin{align*}
\textup{PD}_\textup{t}(12n)\equiv \textup{PD}_\textup{t}(48n)~(\textup{mod}~4),
\end{align*}
which, by iteration, gives \eqref{4k.12n}.

We also extract from \eqref{12n}
\begin{align}
\label{c12}\sum_{n=0}^{\infty}\textup{PD}_\textup{t}(24n+12)q^n&\equiv 2f_4^3-2qf_{12}^3\notag\\
&\equiv2f_{12}a(q^{12})-6q^4f_{36}^3-2qf_{12}^3~(\textup{mod}~4),
\end{align}
from which we extract
\begin{align*}
\sum_{n=0}^{\infty}\textup{PD}_\textup{t}(24(3n+1)+12)q^n&\equiv 2f_{4}^3 -2qf_{12}^3~(\textup{mod}~4).
\end{align*}
From the above two, we have
\begin{align*}
\textup{PD}_\textup{t}(24(3n+1)+12)=\textup{PD}_\textup{t}(3(24n+12))\equiv \textup{PD}_\textup{t}(24n+12)~(\textup{mod}~4).\end{align*}
Thus, for any nonnegative integer $\ell$,
\begin{align*}\textup{PD}_\textup{t}(3^\ell(24n+12))&\equiv \textup{PD}_\textup{t}(24n+12) ~(\textup{mod}~4).\end{align*} Combining the above with \eqref{4k.12n}, we readily arrive at \eqref{24-l-12}.

Next, from \eqref{c11} we also have
\begin{align*}
2\sum_{n=0}^{\infty}\textup{PD}_\textup{t}(24n)q^n&\equiv \dfrac{f_3^2}{f_6}\cdot\dfrac{f_2}{f_1^2}-\dfrac{f_6^2}{f_3^4}+2\notag\\
&\equiv \dfrac{f_3^2}{f_6}\left(\dfrac{f_6^4f_9^6}{f_3^8f_{18}^3}+2q\dfrac{f_6^3f_9^3}{f_3^7}+4q^2\dfrac{f_6^2f_{18}^3}{f_3^6}\right)
-\dfrac{f_6^2}{f_3^4}+2~(\textup{mod}~8),
\end{align*}
from which we extract
\begin{align*}
\sum_{n=0}^{\infty}\textup{PD}_\textup{t}(72n+48)q^n&\equiv 2\dfrac{f_2f_6^3}{f_1^4}\\
&\equiv 2\dfrac{f_6^3}{f_2}~(\textup{mod}~4),
\end{align*}
from which we further extract
\begin{align*}
\textup{PD}_\textup{t}(144n+120)&\equiv 0~(\textup{mod}~4),
\end{align*}
which is \eqref{144-120}.

From \eqref{c12}, we extract
\begin{align}
\label{c13}\sum_{n=0}^{\infty}\textup{PD}_\textup{t}(48n+12)q^n&\equiv 2f_2^3~(\textup{mod}~4)\intertext{and}
\label{c14}\sum_{n=0}^{\infty}\textup{PD}_\textup{t}(48n+36)q^n&\equiv 2f_6^3~(\textup{mod}~4),
\end{align}
which readily implies
\begin{align*}
\textup{PD}_\textup{t}(96n+60)&\equiv 0~(\textup{mod}~4)\\\intertext{and}
\textup{PD}_\textup{t}(96n+84)&\equiv 0~(\textup{mod}~4),
\end{align*}
which are \eqref{96-60} and \eqref{96-84}, respectively. Furthermore, equating the coefficients of $q^{3n+1}$ and $q^{3n+2}$ from both sides of \eqref{c14}, we arrive at
\begin{align*}
\textup{PD}_\textup{t}(144n+84)&\equiv 0~(\textup{mod}~4)\\\intertext{and}
\textup{PD}_\textup{t}(144n+132)&\equiv 0~(\textup{mod}~4),
\end{align*}
which are \eqref{144-84} and  \eqref{144-132}, respectively.

From \eqref{c13} and \eqref{c14}, we also have
\begin{align*}
\sum_{n=0}^{\infty}\textup{PD}_\textup{t}(96n+12)q^n\equiv \sum_{n=0}^{\infty}\textup{PD}_\textup{t}(288n+36)q^n\equiv 2f_1^3\equiv 2\left(f_3a(q^3)-3qf_9^3\right)~(\textup{mod}~4),
\end{align*}
from which we extract
\begin{align*}
\textup{PD}_\textup{t}(288n+204)\equiv \textup{PD}_\textup{t}(3(288n+204))\equiv 0~(\textup{mod}~4),
\end{align*}
which, by iteration, yields \eqref{288-204}.

From \eqref{c11}, we extract
\begin{align*}
\sum_{n=0}^{\infty}\textup{PD}_\textup{t}(48n+24)q^n&\equiv \dfrac{f_2f_3f_4f_{12}}{f_1^3f_6}-2q\dfrac{f_6^2f_{12}^4}{f_3^8}\notag\\
&\equiv \dfrac{f_2f_4f_{12}}{f_6}\cdot\dfrac{f_3}{f_1^3}-2qf_{12}^3\notag\\
&\equiv \dfrac{f_2f_4f_{12}}{f_6}\left(\dfrac{f_4^6f_6^3}{f_2^9f_{12}^2}+3q\dfrac{f_4^2f_6f_{12}^2}{f_2^7}\right)-2qf_{12}^3~(\textup{mod}~4),
\end{align*}
from which we further extract
\begin{align}
\label{96-24}\sum_{n=0}^{\infty}\textup{PD}_\textup{t}(96n+24)q^n&\equiv \dfrac{f_2^7f_3^2}{f_1^8f_6}\notag\\
&\equiv \dfrac{f_3^2}{f_6}\cdot f_2^3\notag\\
&\equiv \dfrac{f_3^2}{f_6} \left(f_6a(q^6)-3q^2f_{18}^3\right) ~(\textup{mod}~4)\intertext{and}
\label{96-72}\sum_{n=0}^{\infty}\textup{PD}_\textup{t}(96n+72)q^n&\equiv 3\dfrac{f_2^3f_6^3}{f_1^6}-2f_6^3\notag\\
&\equiv 2f_6^3+3f_2f_6^3\cdot\dfrac{1}{f_1^2}\notag\\
&\equiv 2f_6^3+3f_2f_6^3\left(\dfrac{f_8^5}{f_2^5f_{16}^2}+2q\dfrac{f_4^2f_{16}^2}{f_2^5f_8}\right)~(\textup{mod}~4).
\end{align}
From \eqref{96-24} we extract
\begin{align*}
\sum_{n=0}^{\infty}\textup{PD}_\textup{t}(288n+216)q^n&\equiv f_6^3\cdot\dfrac{f_1^2}{f_2}\\
&\equiv f_6^3\left(\dfrac{f_9^2}{f_{18}}-2q\dfrac{f_3f_{18}^2}{f_6f_9}\right)~(\textup{mod}~4),
\end{align*}
from which we extract
\begin{align*}
\textup{PD}_\textup{t}(864n+792)&\equiv 0~(\textup{mod}~4),
\end{align*}
which is \eqref{864-792}. We further extract
\begin{align*}
\sum_{n=0}^{\infty}\textup{PD}_\textup{t}(864n+216)q^n&\equiv \dfrac{f_3^2}{f_{6}}\cdot f_2^3\\
 &\equiv \dfrac{f_3^2}{f_{6}}\left(f_6a(q^6)-3q^2f_{18}^3\right)~(\textup{mod}~4),
\end{align*}
from which we have
\begin{align*}
\textup{PD}_\textup{t}(2592n+1080)&\equiv 0~(\textup{mod}~4),
\end{align*}which is \eqref{2592-1080}.

From \eqref{96-72} we extract
\begin{align*}
&\sum_{n=0}^{\infty}\textup{PD}_\textup{t}(192n+72)q^n\\
&\equiv 2f_3^3+3 \dfrac{f_3^3f_4^5}{f_1^4f_8^2}\\
&\equiv 2f_3^3+3f_3^3\cdot \dfrac{f_4}{f_2^2}\\
&\equiv 2f_3^3+3f_3^3\left(\dfrac{f_{12}^4f_{18}^6}{f_6^8f_{36}^3}+2q^2\dfrac{f_{12}^3f_{18}^3}{f_6^7}+4q^4\dfrac{f_{12}^2f_{36}^3}{f_6^6}\right) ~(\textup{mod}~4),
\end{align*}
from which we extract
\begin{align*}
\sum_{n=0}^{\infty}\textup{PD}_\textup{t}(576n+72)q^n&\equiv 2f_1^3+3\dfrac{f_1^3f_{4}^4f_{6}^6}{f_2^8f_{12}^3}\\
&\equiv \left(2+3\dfrac{f_6^2}{f_{12}}\right)\cdot f_1^3\\
&\equiv\left(2+3\dfrac{f_6^2}{f_{12}}\right)\left(f_3a(q^3)-3qf_9^3\right)~(\textup{mod}~4),
\end{align*}
from which we further extract
\begin{align*}
\textup{PD}_\textup{t}(1728n+1224)&\equiv 0~(\textup{mod}~4),
\end{align*}which is \eqref{1728-1224}.

From \eqref{pdt-6n}, we extract
\begin{align}
\label{12n-6}\sum_{n=0}^{\infty}\textup{PD}_\textup{t}(12n+6)q^n&\equiv -\dfrac{f_2f_3f_4f_{12}}{f_1^3f_6}-2\dfrac{f_2^4f_3^4f_4^4}{f_1^{12}f_6^2}\notag\\
&\equiv 2f_4^3+3\dfrac{f_2f_4f_{12}}{f_6}\cdot\dfrac{f_3}{f_1^3}\notag\\
&\equiv 2f_4^3+3\dfrac{f_2f_4f_{12}}{f_6}\left(\dfrac{f_4^6f_6^3}{f_2^9f_{12}^2}+3q\dfrac{f_4^2f_6f_{12}^2}{f_2^7}\right)\notag\\
&\equiv \left(2+3\dfrac{f_6^2}{f_{12}}\right)\cdot f_4^3+qf_{12}^3\cdot\dfrac{f_2^2}{f_4}\\
&\equiv \left(2+3\dfrac{f_6^2}{f_{12}}\right)\left(f_{12}a(q^{12})-3q^4f_{36}^3\right)\notag\\
&\quad+qf_{12}^3\left(\dfrac{f_{18}^2}{f_{36}}-2q^2\dfrac{f_6f_{36}^2}{f_{12}f_{18}}\right)
~(\textup{mod}~4),\notag
\end{align}
from which we extract
\begin{align*}
\textup{PD}_\textup{t}(36n+30)&\equiv 0
~(\textup{mod}~4),
\end{align*}
which is \eqref{36-30}, and
\begin{align*}
\sum_{n=0}^{\infty}\textup{PD}_\textup{t}(36n+18)q^n&\equiv 2qf_{12}^3+3qf_{12}^3\cdot\dfrac{f_2^2}{f_4}+\dfrac{f_6^2}{f_{12}}\cdot f_4^3\\
&\equiv 2qf_{12}^3+3qf_{12}^3\left(\dfrac{f_{18}^2}{f_{36}}-2q^2\dfrac{f_6f_{36}^2}{f_{12}f_{18}}\right)\\
&\quad+\dfrac{f_6^2}{f_{12}}\left(f_{12}a(q^{12})
-3q^4f_{36}^3\right)~(\textup{mod}~4).
\end{align*}
From the above we extract
\begin{align*}
\textup{PD}_\textup{t}(108n+90)&\equiv 0
~(\textup{mod}~4),
\end{align*}
which is \eqref{108-90}, and
\begin{align}
\label{108n-54}\sum_{n=0}^{\infty}\textup{PD}_\textup{t}(108n+54)q^n&\equiv \left(2+3\dfrac{f_6^2}{f_{12}}\right)\cdot f_4^3+qf_{12}^3\cdot\dfrac{f_2^2}{f_4}~(\textup{mod}~4).
\end{align}
From \eqref{12n-6} and \eqref{108n-54}, we have
\begin{equation*}
\textup{PD}_\textup{t}(9(12n+6))\equiv \textup{PD}_\textup{t}(12n+6)~(\textup{mod}~4),
\end{equation*}
which, upon iteration, yields \eqref{12-6}. This completes the proof.

\section*{Acknowledgment} The authors would like to thank the referee for his/her helpful comments and suggestions which helped improving the presentation of the paper.

\end{document}